\documentclass[12pt,reqno]{amsart}
\usepackage[usenames]{color}
\usepackage{courier}

\usepackage{amssymb,amsmath,amsfonts,latexsym,amsthm}
\usepackage{graphicx}
\usepackage{psfrag}
\usepackage{url}		

\calclayout

\DeclareMathOperator{\sgn}{sgn}
\DeclareMathOperator{\Hn}{H}

\theoremstyle{plain}
\newtheorem{theorem}{Theorem}[section]
\newtheorem{proposition}{Proposition}
\newtheorem{corollary}[theorem]{Corollary}
\newtheorem{lemma}[theorem]{Lemma}

\newtheorem{problem}[theorem]{Problem}
\theoremstyle{definition}

\usepackage{multirow}

\begin{document}
\setcounter{page}{1}

\title{Fibonacci--Theodorus Spiral and its properties}
\author[M. R. Bacon]{Michael R. Bacon}
\address{Sumter, South Carolina, U.S.A}
\email{baconmr@gmail.com}
\author[C. K. Cook]{Charles K. Cook}
\address{Professor emeritus \\
               University of South Carolina Sumter\\
               Sumter, South Carolina, U.S.A}
\email{charliecook29150@aim.com}
\author[R. Fl\'orez]{Rigoberto Fl\'orez}
\thanks{The third author was partially supported by The Citadel Foundation.}
\address{Department of Mathematical Sciences\\
		The Citadel\\
		Charleston, SC\\
		U.S.A.}
  \email{rflorez1@citadel.edu}
\author[R. A. Higuita]{Robinson A. Higuita}
 \address{Instituto de Matem\'aticas\\
        Universidad de Antioquia\\
        Medell\'in}
        \email{robinson.higuita@udea.edu.co}
        \author[F. Luca]{Florian Luca}
	\address{Mathematics Division, Stellenbosch University, Stellenbosch, South Africa and Centro de Ciencias Matem\'aticas UNAM, Morelia, Mexico}
	\email{fluca@sun.ac.za}
\author[J. L. Ram\'{\i}rez]{Jos\'e L. Ram\'{\i}rez}
\address{Departamento de Matem\'aticas,  Universidad Nacional de Colombia,  Bogot\'a, Colombia}
\email{jlramirezr@unal.edu.co}

\begin{abstract}
Inspired by the ancient spiral constructed by the Greek philosopher Theodorus, which is based on a concatenated sequence of right triangles, we have developed a new spiral. This spiral, called the \emph{Fibonacci--Theodorus}, features triangle sides with lengths corresponding to Fibonacci numbers. This concept can be naturally generalized to other second-order recurrence relations.

Our exploration of the Fibonacci--Theodorus spiral connects classic results from ancient geometry with the study of sequences and their asymptotic behavior. We examine sequential properties such as the areas, perimeters, and angles of the triangles forming the spiral. Notably, we establish a relationship between the ratio of two consecutive areas and the golden ratio---a pattern that also extends to angles sharing a common vertex. Additionally, we present several asymptotic results, demonstrating, for example, that the sum of the first $n$ areas in the spiral approaches a multiple of the sum of the initial $n$ Fibonacci numbers.

Finally, Hahn, in his work \cite{Hahn}, observed a potential connection between the golden ratio and the ratio of areas between spines of lengths $\sqrt{F_{n+1}}$ and $\sqrt{F_{n+2}-1}$, and those between spines of lengths $\sqrt{F_{n}}$ and $\sqrt{F_{n+1}-1}$ in the Theodorus spiral. However, his work did not provide a formal proof of this conjecture. In this paper, we present a proof for Hahn's conjecture.
\end{abstract}

\maketitle
\emph{Keywords: }
\noindent 
Theodorus Spiral, Fibonacci number, Hahn conjecture, area of a spiral, perimeter of a spiral. 


\section{Introduction}
Theodorus, a Greek philosopher and member of the Pythagorean Society, lived from 465 BC to 398 BC. His significant contributions to the development of irrational numbers, particularly in the context of square roots such as $\sqrt{3}, \sqrt{5}, \ldots, \sqrt{17}$ (as referenced in \cite[p. 95]{ef}), are highly regarded by mathematicians; see, for example, the ``A historical digression" by Hardy and Wright \cite[p. 42]{Hardy}. 
Plato notably mentioned Theodorus and his work on square roots, as recorded in the following passage: 
\begin{quote}
   [Theodorus] was proving to us a certain thing about square roots, I mean the side (i.e. root) of a square of three square units and of five square units, that these roots are not commensurable in length with the unit length, and he went on in this way, taking all the separate cases up to the root of seventeen square units, at which point, for some reason, he stopped. (See \cite[pp. 203-204 and pp. 209-212]{Heath} and  \cite{ef, Plato}.) 
\end{quote} 
	\begin{figure}[ht]
		\centering
		\includegraphics[scale=1.1]{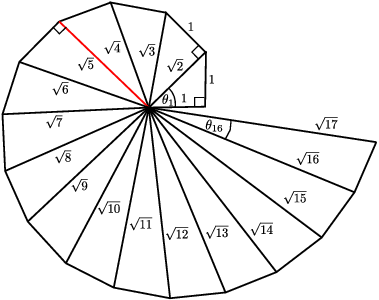}
		 \caption{The Spiral of Theodorus.}
		 \label{TheodorusSpiral} 
	\end{figure}

There is a conjecture that Theodorus may have ceased his investigations due to the absence of a general method for tackling the problem (as discussed in \cite[p138]{Shanks}).

The Pythagorean beliefs concerning irrational numbers are steeped in mythology, including the infamous tale of the drowning of Hippasus for his discoveries.

\begin{quote} The \emph{spiral of Theodorus}, as depicted in Figure \ref{TheodorusSpiral},  is a geometric spiral formed by interconnected right-angled triangles. Each triangle in the sequence has a height of 1 and a base constructed from the hypotenuse of the previous triangle. The initial triangle in the sequence is an isosceles triangle with a length of one unit. This spiral, known as the square root spiral, Einstein spiral, or Pythagorean spiral, was first constructed by Theodorus of Cyrene. (See \cite{Tyler}).
\end{quote}

In our discussion, we refer to the adjacent/hypotenuse sides of the spiral as the \emph{spines}. 
For example, in Figure \ref{TheodorusSpiral}, we highlight the spine labeled with $\sqrt{5}$, which plays a dual role as the hypotenuse for the right-hand triangle and as a leg for the left-hand triangle.

The spiral of Theodorus has been the subject of various investigations, 
and some questions regarding its behavior have been answered, see for example, \cite{Hahn}. Additionally, Crilly \cite{Crilly} presents a generalization of Theodorus, and his approach is accessible to undergraduate students.
One aspect of interest is the behavior of the angle of the $n$-th triangle as $n$ approaches infinity. From the triangles in Figure \ref{TheodorusSpiral}, it is evident that the angle, denoted by $\theta_n$, approaches zero as $n$ tends to infinity.
Thus, 
$$\lim_{n\to\infty}\tan(\theta_n)=\lim_{n\to\infty}\frac{1}{\sqrt{n}}=0.$$
After overlapping two consecutive spines, it is evident that the differences tend to zero as $n$ tends to infinity. Thus, $$\lim_{n\to\infty}(\sqrt{n+1}-\sqrt{n})=0.$$

Surprisingly, as $n$ increases and $\theta_n$ decreases, it takes a greater number of triangles to complete a full revolution. However, the distance between windings of the spiral approaches $\pi$. A simple explanation for this can be rationalized as follows: when $n$ is sufficiently large, $\sin(1/{\sqrt{n}})$ approximates $1/{\sqrt{n}}$. Consequently, it takes at least $\lceil 2\pi\sqrt{n}\,\rceil$ triangles to complete one full winding. Therefore, the approximate distance between windings can be expressed as $\sqrt{n+\lceil 2\pi\sqrt{n}\,\rceil}-\sqrt{n}$. The L'H\^{o}pital's  rule gives that the limit of this expression is 
$$\lim_{x\to\infty}(\sqrt{x+2\pi\sqrt{x}}-\sqrt{x})=\pi.$$ In other words, the distance between windings of the spiral tends towards $\pi$.
For a more detailed and thorough explanation, see \cite[pp. 21-22]{Hahn}.

Hahn  \cite{Hahn} observed a potential connection between the golden ratio and the ratio of areas between spines of lengths $\sqrt{F_{n+1}}$ and $\sqrt{F_{n+2}-1}$. Additionally, he examined the areas between spines of lengths $\sqrt{F_{n}}$ and $\sqrt{F_{n+1}-1}$ in the Theodorus spiral, where Fibonacci numbers are defined as  $F_{n+1}=F_n+F_{n-1}$ for $n\geq 2$, with $F_1=1$ and $F_2=1$. 
However, no formal proof has been provided in his work (see \cite{Hahn}). In this paper, we prove Hahn's conjecture. 

We construct a spiral analogous to that of Theodorus, extending the construction to a general second-order recurrence relation. While our primary focus is on the Fibonacci sequence, we discuss the general case at the end of the paper. For simplicity, we often use the term `Fibonacci spiral' when no ambiguity arises. Although this approach can be applied to any second-order recurrence relation, such as Pell, Lucas, or Jacobsthal sequences, our main focus remains on the Fibonacci sequence, with encouragement for readers to explore similar spirals based on other sequences.

The process involves a particular triangle, as illustrated on the left and in the middle of Figure \ref{FibonacciTheodorusSpiral}. For $n\geq 1$, this triangle exhibits an adjacent side of $\sqrt{F_{n+1}}$, an opposite side of $\sqrt{F_{n}}$, and a hypotenuse of $\sqrt{F_{n+2}}$. Connecting each of these triangles, as demonstrated in the figure, we achieve the desired Fibonacci--Theodorus spiral. The right-hand side of the figure provides a depiction of the Fibonacci--Theodorus spiral construction, showcasing its foundation in the general construction method.

    \begin{figure}[h!]
     \centering	 \includegraphics[width=3.5in]{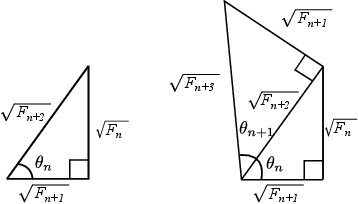} \hspace{1cm}
  \includegraphics[scale = 0.3]{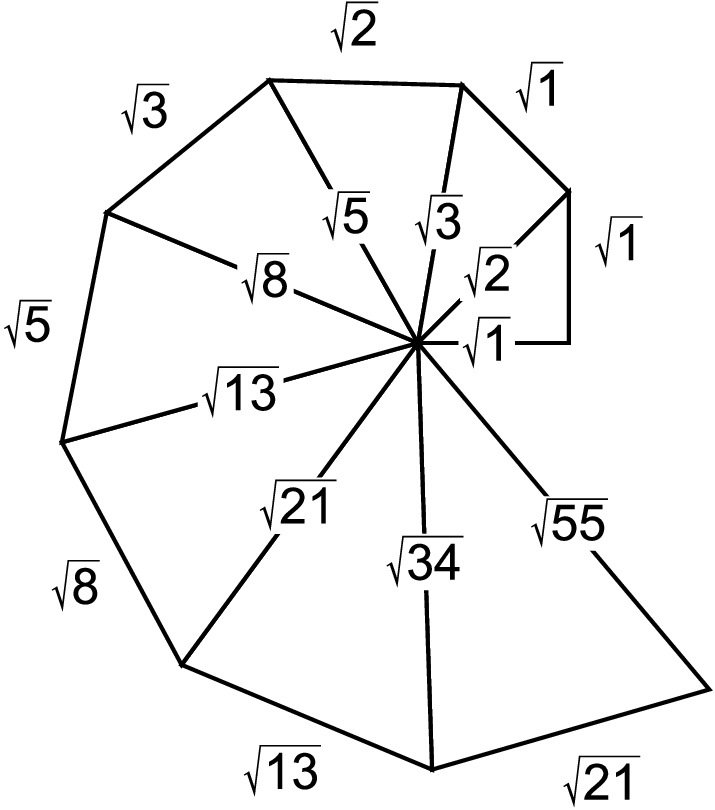}
	\caption{Fibonacci--Theodorus Spiral.}
	 \label{FibonacciTheodorusSpiral} 
    \end{figure}

The nature of the spiral allows us to explore classic questions about triangles. Throughout the paper, we present the area of the circumscribed circle, the ratio between the areas of the circumscribed and inscribed circles, among other classic geometric results.  

\section{Area of the Fibonacci--Theodorus spiral}

In this section, we demonstrate that the ratio of two consecutive areas of the Fibonacci spiral converges to the golden ratio $\varphi=(1+\sqrt{5})/2$. In the second major result of this section we show a relationship between the sum of the first consecutive areas and the Fibonacci numbers.

Let $\Delta_n$ represent the $n$-th triangle of the spiral depicted on the left-hand side of Figure \ref{FibonacciTheodorusSpiral}, with legs $\sqrt{F_n}$, $\sqrt{F_{n+1}}$, and hypotenuse $\sqrt{F_{n+2}}$. We use $\theta_n$ to denote the acute angle located at the common vertex in the spiral. Here $\sqrt{F_n}$ is the opposite leg to $\theta_n$ in $\Delta_n$. 
Each triangle can be represented as an ordered triple: (length of the adjacent side, length of the opposite side, length of the hypotenuse). Therefore, $\Delta_n = (\sqrt{F_{n+1}}, \sqrt{F_n}, \sqrt{F_{n+2}})$. 

We denote the area of each triangle $\Delta_n$ as $A_n$. For example, from the right-hand side of Figure \ref{FibonacciTheodorusSpiral}, we can observe that:
$$A_1=\frac{\sqrt{1\cdot 1}}{2},\quad A_2=\frac{\sqrt{1\cdot 2}}{2},\quad A_3=\frac{\sqrt{2\cdot 3}}{2},\quad A_4=\frac{\sqrt{3\cdot 5}}{2},\quad \dots, \quad A_n=\frac{\sqrt{F_n \cdot F_{n+1}}}{2}.$$ 

In Theorem \ref{RatioArea}, we establish a clear connection between the ratio of two consecutive areas and the divine proportion, $\varphi$.
In this paper, we adopt the symbols  $\varphi$ and  $\beta$ to represent $\left(1+\sqrt{5}\right)/2$ and $\left(1-\sqrt{5}\right)/2$, respectively.
It is important to note that Kepler discovered this  Pythagorean identity $\varphi^2=\varphi+1$. 
Whenever we use this identity within this paper will be made under the name ``Kepler's identity". 
 There is a beautiful relation between the golden ratio and Fibonacci numbers: $\lim_{n\to \infty} {F_{n+1}}/{F_{n}}= \varphi$. Here we leverage this expression to establish another relation between the golden ratio and the ratio of the two consecutive areas of the spiral.  

\begin{theorem}\label{RatioArea}  
$\lim_{n\to \infty} {A_{n+1}}/{A_{n}}= \varphi.
$
\end{theorem}

\begin{proof} From the definition of $A_n$ and after some simplification we have  
\begin{equation*}
  \lim_{n\to \infty} \frac{A_{n+1}}{A_{n}}=\sqrt{\lim_{n\to \infty} \frac{F_{n+2} }{F_{n}}}=\sqrt{\varphi^2}. \qedhere 
\end{equation*}
\end{proof}

\begin{corollary}     
$\lim_{n\to \infty} \theta_n=\csc^{-1}(\varphi)$.
\end{corollary}

\begin{proof} 
Since $\theta_n= \csc^{-1}(\sqrt{F_{n+2}/F_n})$, the result follows taking the limit and then using the previous theorem.
\end{proof}

\begin{proposition}\label{Lema1AreaProp} Let $T(n)=2 \left(\left\lfloor \frac{n-1}{2}\right\rfloor +1\right)-n =  n-2 \left\lfloor \frac{n}{2}\right\rfloor$. Then these hold: 
\begin{enumerate}
\item \label{Lema1AreaPropPart1}
$$
\sum_{k=1}^n \dfrac{(-\varphi)^{(-1)^k-k}}{\sqrt{5\varphi}} = \dfrac{(2\varphi)^{T(n)}\varphi^{-n} -1}{\sqrt{5 \varphi^5}}.
$$
\item \label{Lema1AreaPropPart2}
$$
 \sum_{k=1}^n \dfrac{\varphi^{{\big((-1)^{k+1}-2(k+1)\big)/2}}}{\sqrt{5}}=\dfrac{2-\varphi^{-n+T(n)}+\varphi- \varphi^{-n+2-T(n)}}{\sqrt{5\varphi^5}}.  
$$
\item  \label{Lema1AreaPropPart3}
$$\sum_{k=1}^n \varphi^{k+1/2}/\sqrt{5}= \sqrt{\varphi}(F_{n+2}-1)+\beta^n-1/\sqrt{5\varphi^3}.$$
\end{enumerate}
\end{proposition}

\begin{proof}   First, we observe that the Kepler identity is equivalent to $1 - \varphi^{-2} = \varphi^{-1}$. This, combined with the geometric sum, yields:
\begin{equation}\label{Lema1AreaFormula1}
\sum_{k=0}^{a}\varphi^{-2k}=\dfrac{1-\varphi^{-2(a+1)}}{1-\varphi^{-2}}=\varphi\big(1-\varphi^{-2(a+1)}\big).    
\end{equation}
We prove Part \ref{Lema1AreaPropPart1}.
\begin{eqnarray*}
\sum_{k=1}^n \dfrac{(-1)^{k+1}\varphi^{(-1)^k-k}}{\sqrt{5\varphi}}&=& \sum_{k=0}^{\left\lfloor\frac{n-1}{2}\right\rfloor} \frac{ \varphi^{(-1)^{2 k+1}-1 / 2-(2 k+1)}}{\sqrt{5}}-\sum_{k=1}^{\left\lfloor\frac{n}{2}\right\rfloor}\frac{\varphi^{1 / 2} \varphi^{-2 k}}{\sqrt{5}} \\
&=&\sum_{k=0}^{\left\lfloor\frac{n-1}{2}\right\rfloor}  \frac{\varphi^{-2 k}}{\sqrt{5 \varphi^5}}-\sum_{k=0}^{\left\lfloor\frac{n}{2}\right\rfloor-1}\frac{\varphi^{-2 k}}{\sqrt{5\varphi^3}}.\\
\end{eqnarray*}
Since $T(n)=2 \left(\left\lfloor \frac{n-1}{2}\right\rfloor +1\right)-n =  n-2 \left\lfloor \frac{n}{2}\right\rfloor$,  
using the result in \eqref{Lema1AreaFormula1}, we have: 
\begin{eqnarray*}
\sum_{k=1}^n \dfrac{(-1)^{k+1}\varphi^{(-1)^k-k}}{\sqrt{5\varphi}}&=&
\dfrac{\varphi(1-\varphi^{-2(\left\lfloor\frac{n-1}{2}\right\rfloor+1)})}{\sqrt{5 \varphi^5}}-\dfrac{\varphi(1-\varphi^{-2\left\lfloor\frac{n}{2}\right\rfloor})}{\sqrt{5\varphi^3}}\\
&=& \frac{1-\varphi^{-n-T(n)}}{\sqrt{5 \varphi^3}}-\frac{1-\varphi^{-n+T(n)}}{\sqrt{5 \varphi}} \\ 
&=& \dfrac{1-\varphi^{-n-T(n)}-\varphi+\varphi^{-n+T(n)+1}}{\sqrt{5 \varphi^3}}\\
&=&\dfrac{\varphi^{-n}(\varphi^{T(n)+1}-\varphi^{-T(n)}) -\varphi^{-1}}{\sqrt{5 \varphi^3}}
= \dfrac{(2\varphi)^{T(n)}\varphi^{-n} -1}{\sqrt{5 \varphi^5}}.
\end{eqnarray*}

Next, we prove Part \ref{Lema1AreaPropPart2}, again, using the geometric sum: 
\begin{eqnarray*}
    \sum_{k=1}^n \dfrac{\varphi^{-\frac{2+(-1)^k+2k}{2}}}{\sqrt{5}}&=&     \sum_{k=1}^{\left\lfloor\frac{n}{2}\right\rfloor} \dfrac{\varphi^{-2k}}{\sqrt{5\varphi^3}}+\sum_{k=0}^{\left\lfloor\frac{n-1}{2}\right\rfloor} \dfrac{\varphi^{-2k}}{\sqrt{5\varphi^3}} \\ &=& 
    \varphi^{-2}\sum_{k=0}^{\left\lfloor\frac{n}{2}\right\rfloor-1} \dfrac{\varphi^{-2k}}{\sqrt{5\varphi^3}}+  \sum_{k=0}^{\left\lfloor\frac{n-1}{2}\right\rfloor} \dfrac{\varphi^{-2k}}{\sqrt{5\varphi^3}}\\
    &= &\dfrac{1-\varphi^{-n+T(n)}+\varphi^2- \varphi^{-n+2-T(n)}}{\sqrt{5\varphi^5}}.
\end{eqnarray*}

Now, we prove Part \ref{Lema1AreaPropPart3}. It is easy to verify that $\varphi=1/(\varphi-1)$. This implies 
$$\sum_{k=1}^n\dfrac{\varphi^{k+1/2}}{\sqrt{5}}=\sum_{k=0}^{n-1}\dfrac{\varphi^{k+3/2}}{\sqrt{5}}=\sqrt{\dfrac{\varphi^3}{5}}\dfrac{\varphi^n-1}{\varphi-1}=\sqrt{\dfrac{\varphi^5}{5}}(\varphi^n-1)=\sqrt{\dfrac{\varphi}{5}}(\varphi^{n+2}-\varphi^2).$$
To complete the proof, we use these two facts: $\varphi^{n+2}=\sqrt{5}F_{n+2}+\beta^{n+2}$ and $\varphi^2=\sqrt{5}+\beta^2$. Thus,   
$$
\sum_{k=1}^n\dfrac{\varphi^{k+1/2}}{\sqrt{5}}=\sqrt{\dfrac{\varphi}{5}}\left(\sqrt{5}F_{n+2}+\beta^{n+2}-\sqrt{5}-\beta^2\right)=\sqrt{\varphi}(F_{n+2}-1)+\dfrac{\beta^n-1}{\sqrt{5\varphi^3}}.
$$
This completes the proof. 
\end{proof}

\begin{corollary}\label{Lema1Area} For $n\in\mathbb{Z}_{>0}$ these inequalities hold: 
 $$\dfrac{-1}{\sqrt{5 \varphi^5}}\le 
\sum_{k=1}^n \dfrac{(-\varphi)^{(-1)^k-k}}{\sqrt{5\varphi}} 
\quad \mbox{ and } \quad \sum_{k=1}^n \dfrac{\varphi^{{\big((-1)^{k+1}-2(k+1)\big)/2}}}{\sqrt{5}}\le \dfrac{2+\varphi}{\sqrt{5\varphi^5}}.$$
\end{corollary}

\begin{proof}
From Proposition \ref{Lema1AreaProp} Part \ref{Lema1AreaPropPart1} this holds: 
$$
\dfrac{-1}{\sqrt{5 \varphi^5}} \le  \dfrac{(2\varphi)^{T(n)}\varphi^{-n} -1}{\sqrt{5 \varphi^5}} =\sum_{k=1}^n \dfrac{(-\varphi)^{(-1)^k-k}}{\sqrt{5\varphi}} .
$$

From Proposition \ref{Lema1AreaProp} Part \ref{Lema1AreaPropPart2} this holds:  

$$
 \sum_{k=1}^n \dfrac{\varphi^{{\big((-1)^{k+1}-2(k+1)\big)/2}}}{\sqrt{5}}=\dfrac{2-\varphi^{-n+T(n)}+\varphi- \varphi^{-n+2-T(n)}}{\sqrt{5\varphi^5}}  \le  \dfrac{2+\varphi}{\sqrt{5\varphi^5}}.
$$
This completes the proof.
\end{proof}

We use $\sgn(x)$ to represent $x/|x|$, the sign function of $x\ne 0$. 
The next theorem demonstrates that the sum of the first $n$ areas of the triangles within the spiral is of order $F_{n+2}$. 

\begin{theorem} If $n\in\mathbb{Z}_{>0}$, then 
$$\displaystyle{ \Big|\sum_{i=1}^n A_i- \frac{\sqrt{\varphi}}{2}(F_{n+2}-1)\Big|}\le \frac{1}{4}.$$
\end{theorem}

\begin{proof} 
We commence this proof by observing two crucial facts: 
firstly, we have
\begin{eqnarray*}
2A_k&=&\sqrt{F_kF_{k+1}}= \sqrt{\dfrac{\varphi^k-\beta^k}{\sqrt{5}}\times\dfrac{\varphi^{k+1}-\beta^{k+1}}{\sqrt{5}}}.
\end{eqnarray*}
 Since $\beta/\varphi =-\beta ^2$, we have that
\begin{eqnarray*}
2A_k&=&  \sqrt{\varphi^{2k+1}\dfrac{1-(-\beta^2)^k}{\sqrt{5}}\times\dfrac{1-(- \beta^2)^{k+1}}{\sqrt{5}}}\\
&=&\varphi^k \sqrt{\dfrac{\varphi}{5}}\sqrt{(1-(-\beta^2)^k)(1+\beta^2(-\beta^2)^{k})}\\
&=&\varphi^k \sqrt{\dfrac{\varphi}{5}}\sqrt{1+\beta (-\beta^2)^{k}- \beta^2 (\beta^2)^{2k} }.
\end{eqnarray*}

Secondly, if $-\varphi^2\le x\le 1$, then we have 
$$
1-(-\beta)^{1-\sgn(x)}x\le
\sqrt{1+\beta x-\beta x^2}= \sqrt{(1-x)(1+\beta^2 x)}\le 1+\beta^{(\sgn(x)+3)/2}x.
$$
Therefore, if we set $x=(-\beta^2)^k=(-\varphi^2)^{-k}$ and use the fact that $\varphi\beta=-1$, we have
$$
\varphi^k \sqrt{\dfrac{\varphi}{5}}\left(1-(-\beta)^{1-(-1)^k}(-\varphi^2)^{-k} \right)\le
2A_k\le \varphi^k \sqrt{\dfrac{\varphi}{5}}\left(1+\beta^{\frac{(-1)^k+3}{2}}(-\varphi^2)^{-k}\right)
$$
and that 
$$
\varphi^k \sqrt{\dfrac{\varphi}{5}}\left(1-\varphi^{(-1)^k-1}(-\varphi^2)^{-k} \right)\le
2A_k\le \varphi^k \sqrt{\dfrac{\varphi}{5}}\left(1+\varphi^{-\frac{3+(-1)^k}{2}}\varphi^{-2k}\right).
$$
Hence, we obtain
$$
\dfrac{\varphi^{k+1/2}-(-1)^k\varphi^{(-1)^k-1/2-k}}{\sqrt{5}} \le
2A_k\le  \dfrac{\varphi^{k+1/2}+\varphi^{-\frac{2+(-1)^k+2k}{2}}}{\sqrt{5}}.
$$
This, in combination with Corollary \ref{Lema1Area}, implies
$$
\sum_{k=1}^n \dfrac{(-1)^{k+1}\varphi^{(-1)^k-1/2-k}}{\sqrt{5}}\le \sum_{k=1}^n 2A_k-\sum_{k=1}^n\dfrac{\varphi^{k+1/2}}{\sqrt{5}}\le \sum_{k=1}^n \dfrac{\varphi^{-\frac{2+(-1)^k+2k}{2}}}{\sqrt{5}}.
$$
From this previous result, we conclude that 
$$
\dfrac{-1}{\sqrt{5 \varphi^5}}\le \sum_{k=1}^n 2A_k-\left(\sqrt{\varphi}(F_{n+2}-1)+\frac{\beta^n-1}{\sqrt{5\varphi^3}}\right)\le \dfrac{2+\varphi}{\sqrt{5\varphi^5}}. 
$$
Hence, 
$$-\frac{1}{2}\le 
\dfrac{-1}{\sqrt{5 \varphi^5}}+\dfrac{\beta^n-1}{\sqrt{5\varphi^3}}\le \sum_{k=1}^n 2A_k-\sqrt{\varphi}(F_{n+2}-1)\le \dfrac{2+\varphi}{\sqrt{5\varphi^5}}+\dfrac{\beta^n-1}{\sqrt{5\varphi^3}}\le \frac{1}{2}.
$$
This concludes the proof.
\end{proof}

The preceding theorem demonstrates that the initial external perimeter of the Fibonacci spiral is of order $(F_{n+2} -1)$. Our focus was on demonstrating the relationship between the sum of the first areas of the spiral and Fibonacci numbers. 

The following result represents an abstract generalization of Theorem \ref{RatioArea}, where the case $m=1$ corresponds to the statement of the theorem.

In this paper we use the most standard notation for $n$-th Lucas number, $L_n$, where  $L_0=2$, $L_1=1$, and $L_n=L_{n-1}+L_{n-2}$ for $n\geq 2$.
\begin{proposition}
Let 
$$A_{n,m}:=\dfrac{\sqrt{\displaystyle{\prod_{i=0}^m F_{n+i}}}}{2}.$$
Then for all $m>0$, 
$$\lim_{n\to \infty}\dfrac{A_{n+1,m}}{A_{n,m}}=\sqrt{\frac{1}{2}\left(\sqrt{5}F_{m+1}+L_{m+1}\right)},$$
\end{proposition}

\begin{proof} First of all we observe that 
$$
\frac{A_{n+1,m}}{A_{n,m}}=\sqrt{\frac{F_{n+1}F_{n+2}\cdots F_{n+m}F_{n+1+m}}{F_nF_{n+1}\cdots F_{n+m}}}={\sqrt{\frac{F_{n+m+1}}{F_n}}}.
$$
The identity  
$
2F_{a+b}=F_aL_b+F_bL_a,
$
with $a=n,~b=m+1$, gives that 
$$
\frac{F_{n+m+1}}{F_n}=\frac{1}{2}\left(L_{m+1}+\left(\frac{L_n}{F_n}\right)F_{m+1}\right).
$$
This imply that 
\begin{eqnarray*}
\lim_{n\to \infty} \frac{A_{n+1,m}}{A_{n,m}} =  {\sqrt{\frac{F_{n+m+1}}{F_n}}}
 =  {\sqrt{\frac{1}{2}\left(L_{m+1}+\left(\lim_{n\to\infty} \left(\frac{L_n}{F_n}\right)\right) F_{m+1}\right)}}. 
\end{eqnarray*}
The fact that $\lim_{n\to\infty} L_n/F_n={\sqrt{5}}$, completes the desired result. 
\end{proof}

\section{Perimeter of Fibonacci--Theodorus spiral}
In this section, we delve into the analysis of two primary aspects concerning the external perimeter of the Fibonacci spiral.
We use $S_n$ to denote the \emph{initial external perimeter} of the $n$-th Fibonacci spiral, which is formed with precisely $n$ triangles. 

\begin{theorem} \label{Spiral} Let $C=\varphi^2(1+\varphi^{-1/2})$. For $n\in\mathbb{Z}_{>1}$, then 
  $$\left|S_n-C\Big(\sqrt{F_n}-\frac{1}{\sqrt[4]{5}}\Big)\right|\le 1.$$ 
\end{theorem} 

\begin{proof} 
The geometric sum $\sum_{i=1}^n{\varphi^{i/2}}$ simplifies as follows:   
\begin{equation}\label{GeometricSum}
 \sum_{i=1}^n{\varphi^{i/2}}=\dfrac{\varphi^{1/2}(\varphi^{n/2}-1)}{\varphi^{1/2}-1}=\varphi^{1/2}(\varphi^{1/2}+1)\dfrac{\varphi^{1/2}(\varphi^{n/2}-1)}{\varphi-1}
=\varphi^2(1+\varphi^{-1/2})(\varphi^{n/2}-1).   
\end{equation}

On the other hand, from the Binet formula, we have the following for every
$i\in\mathbb{N}$:  
    $$
    \sqrt{F_{i}}=\sqrt{\dfrac{\varphi^i-\beta^i}{\sqrt{5}}}=\sqrt{\dfrac{\varphi^i}{\sqrt{5}}}\left(1-\left(\dfrac{-1}{\varphi^2}\right)^i\right)^{1/2}. 
    $$
This implies that
\begin{equation}\label{RootBounds}
\sqrt{\dfrac{\varphi^i}{\sqrt{5}}}\left(1-\left(\dfrac{1}{\varphi^2}\right)^i\right)
\le \sqrt{F_i}\le \sqrt{\dfrac{\varphi^i}{\sqrt{5}}}\left(1+\left(\dfrac{1}{\varphi^2}\right)^i\right).  
\end{equation}
Summing up the previous inequality and then simplifying, we obtain  
$$\sum_{i=1}^n{\dfrac{\varphi^{i/2}}{\sqrt[4]{5}}}-\sum_{i=1}^n \dfrac{\varphi^{-3/2i}}{\sqrt[4]{5}}
\le  \sum_{i=1}^n \sqrt{F_i}\le \sum_{i=1}^n{\dfrac{\varphi^{i/2}}{\sqrt[4]{5}}}+\sum_{i=1}^n \dfrac{\varphi^{-3/2i}}{\sqrt[4]{5}}.$$ 
Hence, 
$$-\sum_{i=1}^n \dfrac{\varphi^{-3/2i}}{\sqrt[4]{5}}
\le \sum_{i=1}^n \sqrt{F_i} -\sum_{i=1}^n{\dfrac{\varphi^{i/2}}{\sqrt[4]{5}}}\le \sum_{i=1}^n \dfrac{\varphi^{-3/2i}}{\sqrt[4]{5}}.$$
By further simplification, using the geometric sum, we get
\begin{eqnarray*}
-\dfrac{1}{\sqrt[4]{5}(\varphi^{3/2}-1)}\le -\dfrac{1-\varphi^{-3/2n}}{\sqrt[4]{5}(\varphi^{3/2}-1)}
&\le& \sum_{i=1}^n \sqrt{F_i} -\sum_{i=1}^n{\dfrac{\varphi^{i/2}}{\sqrt[4]{5}}}\\ 
&\le& \dfrac{1-\varphi^{-3/2n}}{\sqrt[4]{5}(\varphi^{3/2}-1)} \le \dfrac{1}{\sqrt[4]{5}(\varphi^{3/2}-1)}.
\end{eqnarray*}

This, together with \eqref{GeometricSum}, implies  
$$-\dfrac{1}{\sqrt[4]{5}(\varphi^{3/2}-1)}
\le \sum_{i=1}^n \sqrt{F_i} -\frac{\varphi^2(1+\varphi^{-1/2})(\varphi^{n/2}-1)}{\sqrt[4]{5}} \le \dfrac{1}{\sqrt[4]{5}(\varphi^{3/2}-1)}.$$

Now, let us provide some bounds for 
 $\varphi^2(1+\varphi^{-1/2})(\varphi^{n/2}-1)$
in terms of $F_n$ and $\varphi$. 
We start bounding $\varphi^{n/2}$. From the fact that $\varphi^n=\sqrt{5}F_n+ (-\varphi^{-2})^n$, we have:   
$$
\varphi^{n/2}=\sqrt{\sqrt{5}F_{n}+(-\varphi^{-2})^{n}}= \sqrt{\sqrt{5}F_n}\left( 1+\dfrac{1}{\sqrt{5}(-\varphi^2)^nF_n}\right)^{1/2}.
$$
Using the fact that 
$1-|x|\le \sqrt{1+x}\le 1+|x|$, with $x=\dfrac{1}{\sqrt{5}(-\varphi^2)^nF_n}$, we can imply that    
$$
\sqrt{\sqrt{5}F_n}-\dfrac{1}{\sqrt{5F_n}\varphi^{2n}} \le \varphi^{n/2} \le  \sqrt{\sqrt{5}F_n}+\dfrac{1}{\sqrt{5F_n}\varphi^{2n}}.
$$

Subtracting $1$ in all parts to this inequality and then in multiplying all part of this previous inequality by 
$(\varphi^2(1+\varphi^{-1/2}))/(\sqrt[4]{5})$,  
then we can establish the following bounds:
$$
T_n-C_n\le \frac{\varphi^2(1+\varphi^{-1/2})(\varphi^{n/2}-1)}{\sqrt[4]{5}} \le T_n+C_n, 
$$
where  
$$C_n:=\frac{\varphi^2(1+\varphi^{-1/2})}{\sqrt[4]{5}\sqrt{5F_n}\varphi^{2n}} \quad \text {and} \quad T_n:= \varphi^2(1+\varphi^{-1/2})(\sqrt{F_n}-(\sqrt[4]{5})^{-1}).$$
It is worth to note that $0< C_n \le 0.21$ for $n>1$. Moreover, $C_n$ converges to zero. Therefore, for $n>1$, this gives us  
$$
-\left(\dfrac{1}{\sqrt[4]{5}(\varphi^{3/2}-1)}
+C_n\right)\le \sum_{i=1}^n \sqrt{F_i} - \varphi^2(1+\varphi^{-1/2})(\sqrt{F_n}-(\sqrt[4]{5})^{-1})\le \dfrac{1}{\sqrt[4]{5}(\varphi^{3/2}-1)}+C_n.
$$
This completes the proof. 
\end{proof}

The preceding theorem demonstrates that the initial external perimeter of the Fibonacci spiral is of order $\Big(\sqrt{F_n} -\frac{1}{\sqrt[4]{5}}\Big)$. Our focus was on demonstrating the relationship between the initial external perimeter and Fibonacci numbers. Nevertheless, a more precise approximation can be achieved by employing the Taylor series, as expressed by:
$$\left(1-\left(\dfrac{-1}{\varphi^2}\right)^i\right)^{1/2}= \sum_{n=0}^{\infty}(-1)^n \binom{1/2}{ n}\left(\dfrac{-1}{\varphi^2}\right)^{in}.$$ 
However, we defer the details to the interested reader.

\subsection{Asymptotic behaviors in the Fibonacci-Theodorus spiral} The first result arises from the observation that by taking three segments of the Fibonacci spiral---not necessarily consecutive but with exactly one segment between them---we can obtain an identity that, in a certain sense, mimics the definition of Fibonacci numbers. For instance, in the fourth, sixth, and eighth segments on the left-hand side of Figure \ref{FibonacciTheodorusSpiral}, the lengths of these segments are $\sqrt{3}$, $\sqrt{8}$, and $\sqrt{21}$. They satisfy this equality:
$\lfloor \sqrt{3} + \sqrt{8} \rfloor = \lfloor \sqrt{21} \rfloor$.

\begin{proposition} \label{RootIdentityR}
Let $n$ be a non-negative integer $n\neq 8$. Then 
 $$\left\lfloor \sqrt{F_{n+4}}\right\rfloor =\left\lfloor \sqrt{F_{n+2}}+\sqrt{F_n}\right\rfloor.$$
\end{proposition}

\begin{proof} It is easy to verify that the result holds for $n<8$. 
Let  
$$
\delta_n:={\sqrt{F_{n+4}}}-({\sqrt{F_{n+2}}}+{\sqrt{F_n}}), \qquad x_n:={\sqrt{F_{n+4}}}+{\sqrt{F_{n+2}}}+{\sqrt{F_n}},
$$
and 
$$
 y_n:=F_{n+1}+{\sqrt{F_nF_{n+2}}}.
$$
Then
\begin{eqnarray*}
\delta_n x_n  & = & \left({\sqrt{F_{n+4}}}-({\sqrt{F_{n+2}}}+{\sqrt{F_n}})\right)({\sqrt{F_{n+4}}}+({\sqrt{F_{n+2}}}+{\sqrt{F_n}})) \\
& = & F_{n+4}-({\sqrt{F_{n+2}}}+{\sqrt{F_n}})^2\\
& = & F_{n+4}-(F_{n+2}+F_n+2{\sqrt{F_{n+2}F_n}})\\
& = & (F_{n+3}+F_{n+2}-F_{n+2}-F_n)-2{\sqrt{F_{n+2}F_n}}\\
& = & F_{n+2}+F_{n+1}-F_n-2{\sqrt{F_{n+2}F_n}}\\
& = & 2F_{n+1}-2{\sqrt{F_{n+2}F_n}}.
\end{eqnarray*}
Therefore, 
$$
\delta_n x_n y_n=2(F_{n+1}-{\sqrt{F_{n+2}F_n}})(F_{n+1}+{\sqrt{F_{n+2}F_n}})=2(F_{n+1}^2-F_{n+2}F_n)=2(-1)^n.
$$
So,  
\begin{equation}\label{RootEqua}
\delta_n=\frac{2(-1)^n}{x_n y_n}.
\end{equation}
Using the fact that for $n\ge 1$, 
$$
\varphi^{n-2}\le F_n\le \varphi^{n-1}, 
$$
we have 
\begin{eqnarray*}
x_n & = & {\sqrt{F_{n+4}}}+{\sqrt{F_{n+2}}}+{\sqrt{F_n}}>\varphi^{(n+2)/2}+\varphi^{n/2}+\varphi^{(n-2)/2}\\
& = & \varphi^{(n-2)/2}(\varphi^2+\varphi+1)=2\varphi^{(n+2)/2},
\end{eqnarray*}
while 
$$
y_n=F_{n+1}+{\sqrt{F_{n+2}F_n}}>\varphi^{n-1}+\varphi^{(n+n-2)/2}=2\varphi^{n-1}.
$$

This and \eqref{RootEqua} imply 
$$
|\delta_n|<\frac{2}{(2^2\varphi^{(n+2)/2} \varphi^{n-1})}=\frac{1}{2\varphi^{3n/2}}.
$$

To complete the proof that $\left\lfloor \sqrt{F_{n+4}}\right\rfloor =\left\lfloor \sqrt{F_{n+2}}+\sqrt{F_n}\right\rfloor$, we now consider two cases. 
We prove the case in which $n$ is even (the odd case is similar and we omit it).  The only way in which the previous equality does not hold is when the left--hand side is bigger than the right--hand side. But then, putting $k_n:=\left\lfloor \sqrt{F_{n+4}}\right\lfloor$, we would get 
$$
{\sqrt{F_{n+4}}}=k_n+\zeta_n\quad {\text{\rm and }}\quad {\sqrt{F_{n+2}}}+{\sqrt{F_n}}=(k_n-1)+\eta_n\quad {\text{\rm with }}\quad \zeta_n,\eta_n\in (0,1).
$$
Thus, 
$$
\delta_n=\zeta_n+(1-\eta_n)=\frac{2}{x_ny_n},
$$
which shows that 
$$
0<\zeta_n<\delta_n<\frac{2}{x_n y_n}<\frac{1}{2\varphi^{3n/2}}.
$$
On the other hand, let us write 
$$
F_{n+4}=k_n^2+w_n,\qquad {\text{\rm where}}\qquad 1\le w_n\le 2k_n.
$$
It is worth to note that $F_{12}=12^2$ is the largest Fibonacci that is a prefect square. Thus, $F_{n+4}$ is  not a perfect square when $n>8$. Therefore, $w_n\ne 0$. 
So, 
$$
\zeta_n={\sqrt{F_{n+4}}}-k_n={\sqrt{k_n^2+w_n}}-k_n=\frac{w_n}{k_n+{\sqrt{k_n^2+w_n}}}>\frac{1}{2k_n+1}.
$$
This implies that 
$$
\frac{1}{2k_n+1}<\zeta_n<\delta_n<\frac{1}{2\varphi^{3n/2}},
$$
it follows that  
$$
2\varphi^{3n/2}<2k_n+1<2{\sqrt{F_{n+4}}}+1<2\varphi^{(n+3)/2}+1,
$$
or equivalently, 
$$
2\varphi^{(n+3)/2} (\varphi^{(2n-3)/2}-1)<1.
$$
That is a contradiction for $n>8$, because all three factors ---$\varphi^{(n+3)/2}$, $~\varphi^{(2n-3)/2}-1$, and $2$--- on the left--hand side are greater than $1$, and therefore so is their product.  
\end{proof}

\begin{corollary}
 \[\lim_{n\to \infty} \frac{\sqrt{F_{n+4}}}{\sqrt{F_{n+2}}+\sqrt{F_n}}=1.\]
\end{corollary}

\begin{proof} Using the notation given in the proof of Proposition \ref{RootIdentityR} we have
$$
\frac{\sqrt{F_{n+4}}}{{\sqrt{F_{n+2}}}+{\sqrt{F_n}}}=\frac{\sqrt{F_{n+4}}}{{\sqrt{F_{n+4}}}-\delta_n}=\frac{1}{1-(\delta_n/{\sqrt{F_{n+4}}})}.
$$
Since $\delta_n\to 0$ and ${\sqrt{F_{n+4}}}\to\infty$, the right--hand side above tends to $1$. 
\end{proof}

\begin{proposition}
Let $P_{n}=\sqrt{F_{n}}+\sqrt{F_{n+1}}+\sqrt{F_{n+2}}$ be the sequence of perimeters of the triangles $\Delta_n$ for $n\ge 1$. Then  
$\lim_{n\to \infty}P_{n+1}/P_{n}=\sqrt{\varphi}.$
\end{proposition}

\begin{proof} It is easy to see that 
\begin{eqnarray*}
\frac{P_{n+1}}{P_n}  & = & \frac{{\sqrt{F_{n+1}}}+{\sqrt{F_{n+2}}}+{\sqrt{F_{n+3}}}}{{\sqrt{F_n}}+{\sqrt{F_{n+1}}}+{\sqrt{F_{n+2}}}}\\
& = & {\sqrt{\frac{F_{n+1}}{F_n}}}\cdot \frac{1+{\sqrt{F_{n+2}/F_{n+1}}}+{\sqrt{F_{n+3}/F_{n+1}}}}{1+{\sqrt{F_{n+1}/F_n}}+{\sqrt{F_{n+2}/F_n}}}.
\end{eqnarray*}
Since $\lim_{n\to\infty} F_{n+1}/F_n=\varphi$, we have 
$$\lim_{n\to \infty}P_{n+1}/P_{n} = {\sqrt{\varphi}} \cdot \frac{1+{\sqrt{\varphi}}+{\sqrt{\varphi^2}}}{1+{\sqrt{\varphi}}+{\sqrt{\varphi^2}}}={\sqrt{\varphi}}.$$
This completes the proof
\end{proof}

In classic geometry, we know that the \emph{centroid of a triangle} is the point at which the three medians intersect. Let $C_n$ be the centroid of the  triangle $\Delta_n$. In Figure \ref{centroida}, we show the centroids for the first triangles in the Fibonacci spiral.  For $n\geq 2$, let $d_n$ be the sum of the distances between the points $C_1, C_2, \dots, C_n$. Denote by $c_i$ the distance between $C_{i-1}$ and $C_i$. 
We can verify that $$d_2=\frac{\sqrt{2+\sqrt{2}}}{3}, \quad d_3=\frac{1}{9} \bigg(3 \sqrt{2+\sqrt{2}}+\sqrt{3 \left(9+4 \sqrt{3}\right)}\bigg).$$

  \begin{figure}[ht]
\centering
 \includegraphics[scale = 0.5]{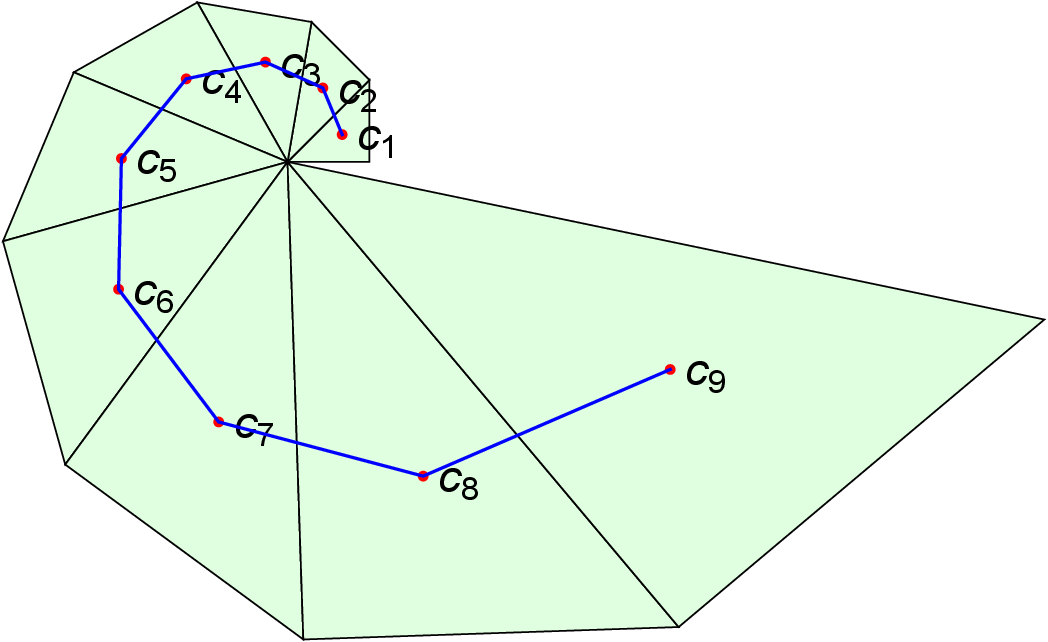} 
	\caption{Centroids in the Fibonacci spiral.} 
	 \label{centroida}
\end{figure}

\begin{lemma} \label{Problem5} If $n\ge 2$, then 
$$c_n=\frac{1}{3} {\sqrt{F_{n+1}+2F_n{\sqrt{F_{n-1}/F_{n+1}}}}}.$$ 
\end{lemma}

\begin{proof}
Using the notation from Figure \ref{distanceConsecutivecentroids}, let $M_1$ and $M_2$ represent the midpoints of $BC$ and $BD$, respectively. Additionally, $G_1$ and $G_2$ are the centroids of triangles $\triangle ABC$ and $\triangle ABD$.

Since the ratios $AG_1/AM_1$ and  $AG_2/AM_2$ are both equal to $2/3$, we conclude that $G_1G_2$ is parallel to $M_1M_2$, and the triangles $\triangle AG_1G_2$ and $\triangle AM_1M_2$ are similar. Specifically, we have $G_1G_2/M_1M_2 = 2/3$. To find $M_1M_2$, we apply the cosine theorem to triangle $\triangle AM_1M_2$, where the relevant data are given by
$$
BM_1=BC/2={\sqrt{F_{n-1}}}/2,~BM_2=BD/2={\sqrt{F_n}}/2,
$$
and 
$$\cos(\angle CBD)=\cos(90^\circ+\angle CBA)=-\sin(\angle CBA)=-AC/AB=-\sqrt{F_n/F_{n+1}}.$$ 
Using these values, we calculate $M_1M_2^2$ as follows:
\begin{eqnarray*}
M_1M_2^2 & = & BM_1^2+BM_2^2-2BM_1BM_2\cos(\angle CBD)\\
& = & F_{n-1}/4+F_n/4+2({\sqrt{F_{n-1}}}/2)({\sqrt{F_n}}/2) {\sqrt{F_n/F_{n+1}}}\\
& = & (F_{n-1}+F_n)/4+(F_n/2){\sqrt{F_{n-1}/F_{n+1}}}\\
& = & F_{n+1}/4+(F_n/2){\sqrt{F_{n-1}}/F_{n+1}}.
\end{eqnarray*}
Thus, 
$$M_1M_2=\frac{1}{2}\sqrt{F_{n+1}+2F_n {\sqrt{F_{n-1}/F_{n+1}}}}.
$$
Finally, since $c_n=G_1G_2=(2/3)M_1M_2$, the desired formula follows.
\end{proof}

  \begin{figure}[ht]
\centering
 \includegraphics[scale = 1.5]{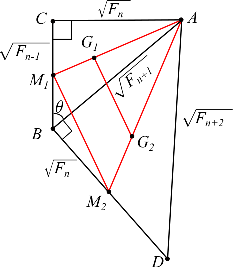} 
	\caption{Distance between to consecutive centroids.} 
	 \label{distanceConsecutivecentroids}
\end{figure}

\begin{theorem} \label{Problem5} If $n\ge 2$, then 
  $\lim_{n\to \infty}d_{n+1}/d_{n}=\sqrt{\varphi}$.
\end{theorem}

\begin{proof}
From Lemma \ref{Problem5} we know 
\begin{equation*}
\label{eq:cn}
c_n=\frac{1}{3} {\sqrt{F_{n+1}+2F_n{\sqrt{F_{n-1}/F_{n+1}}}}}
\end{equation*}
for $n\ge 2$. For $n=2, 3$ we get
$$
c_2=\frac{1}{3}{\sqrt{F_3+2F_2{\sqrt{F_1/F_3}}}}=\frac{1}{3}{\sqrt{2+{\sqrt{2}}}},
$$
and 
\begin{eqnarray*}
c_3 & = & \frac{1}{3}{\sqrt{F_4+2F_3{\sqrt{F_2/F_4}}}}\\
& = & \frac{1}{3}{\sqrt{3+2\cdot 2{\sqrt{1/3}}}}=\frac{1}{3}{\sqrt{3+4/{\sqrt{3}}}}\\
& =& \frac{1}{3{\sqrt{3}}}{\sqrt{3{\sqrt{3}}+4}}=\frac{\sqrt{4{\sqrt{3}}+9}}{9}.
\end{eqnarray*}
Using formula \eqref{eq:cn} we get the asymptotic 
\begin{eqnarray*}
c_n & =&  \frac{1}{3} {\sqrt{\frac{\varphi^{n+1}}{\sqrt{5}}(1+o(1))+2\frac{\varphi^n}{{\sqrt{5}}} {\sqrt{\frac{1}{\varphi^2}}}(1+o(1))}}\\
& = & \frac{1}{3} {\sqrt{\frac{\varphi^{n+1}}{{\sqrt{5}}}(1+o(1))+\frac{2\varphi^{n-1}}{{\sqrt{5}}}(1+o(1))}}\\
& = & \frac{{\sqrt{\varphi^2+2}}}{3\cdot 5^{1/4}}\varphi^{(n-1)/2}(1+o(1)).
\end{eqnarray*}
The $o(1)$ above tends to zero exponentially. So, it would seem that 
\begin{eqnarray*}
d_n & = & \sum_{2\le m\le n} c_m=\sum_{{\sqrt{n}}\le m\le n} \frac{{\sqrt{\varphi^2+2}}}{3\cdot 5^{1/4}}\varphi^{(m-1)/2}(1+o(1))+O(\varphi^{\sqrt{n}})\\
& = & \frac{{\sqrt{\varphi^2+2}}}{3\cdot 5^{1/4}({\sqrt{\varphi}}-1)} \varphi^{n/2} (1+o(1)).
\end{eqnarray*}
This asymptotic implies the statement about the limit of $d_{n+1}/d_n$ when $n$ tends to infinity. 
\end{proof}

\section{Proof of Hahn's conjecture}\label{S-hahn}

In this section, we introduce a specific concept referred to as the $n$-th \emph{Hahn's area}, denoted by $\Hn(n)$. This designation pays tribute to Hahn, who initially introduced this concept in \cite{Hahn}. As illustrated in Figure \ref{centroid}, this concept involves the area of $F_{n}$ triangles within the Theodorus spiral, starting from the $(F_{n+1})$-th triangle. Thus, $\Hn(n)=\sum_{i=0}^{F_n-1} \sqrt{F_{n+1}+i}/2$. 

For instance, to find $\Hn(5)$, we sum the areas of the $(F_6)$-th, $(F_6+1)$-th, $(F_6+2)$-th, $(F_6+3)$-th, and $(F_6+F_5-1)$-th triangles in the Theodorus spiral (see  Figure  \ref{centroid}). Since one leg has a length equal to one and the other leg is the square root of a natural number, these areas, denoted by $M_i$, are $M_8=\sqrt{F_6}/2$, $M_9=\sqrt{F_6+1}/2$, $M_{10}=\sqrt{F_6+2}/2$, $M_{11}=\sqrt{F_6+3}/2$, and $M_{12}=\sqrt{F_6+F_5-1}/2$. 

It is worth noting that Figure \ref{centroid} does not correspond to the Fibonacci-Theodorus spiral, it corresponds to the Spiral of Theodorus.

\begin{figure}[ht]
\centering
  \includegraphics[scale = 0.45]{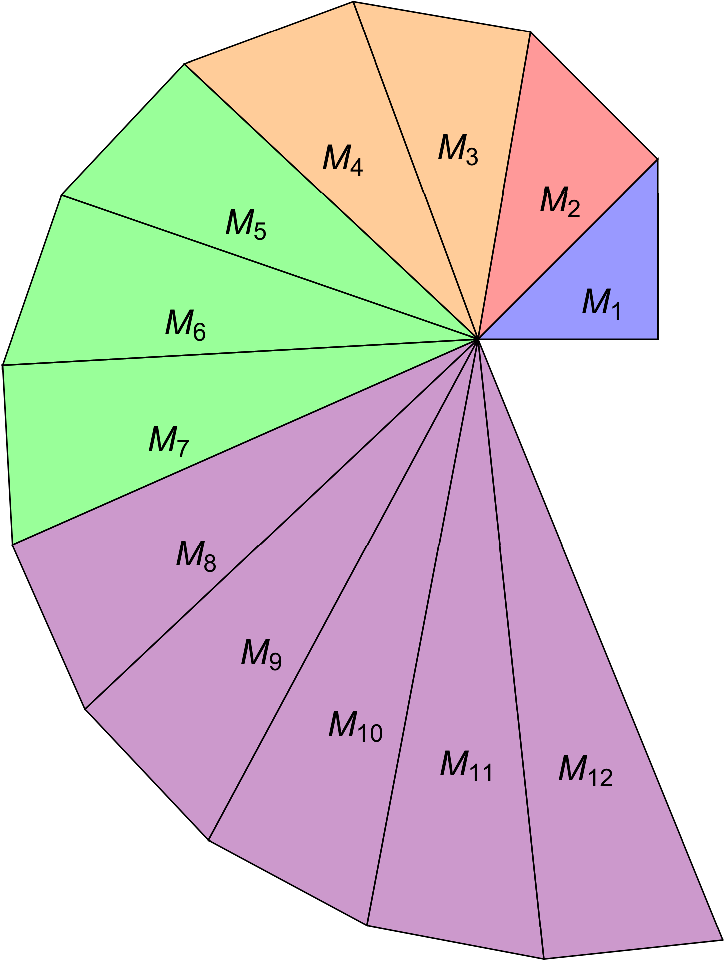}
	\caption{Hahn's conjecture.} 
	 \label{centroid}
\end{figure}

Hahn, in his work \cite{Hahn}, proposed an interesting conjecture concerning the ratio of two consecutive Hahn areas. He suggested that this ratio converges to $\varphi\sqrt{\varphi}$. More precisely,  he proposed that 
$\lim_{n\to \infty} \Hn(n)/\Hn(n-1)=\varphi\sqrt{\varphi}=\sqrt{2+\sqrt{5}}$.

We now prove Hahn's conjecture. First, we introduce the $n$-th \emph{harmonic number} of order $p$ as 
\begin{equation}\label{HarmonicNumber}
  H_{m,p}:=\sum_{k=1}^{m} \frac{1}{k^p}.  
\end{equation}

\begin{theorem} If $n \in \mathbb{Z}$, then 
$$\lim_{n\to \infty} \frac{\Hn(n)}{\Hn(n-1)}=\varphi\sqrt{\varphi}.$$    
\end{theorem}

\begin{proof} 
From \cite{Chlebus} we know that 
\begin{equation}\label{HarmonicInequality}
    1+\dfrac{m^{1-p}-1}{1-p}<H_{m,p} <1+\dfrac{(m+1)^{1-p}-1}{1-p}.
\end{equation}
From \eqref{HarmonicNumber}, we can observe that 
$H_{F_{n+2}-1,{-1/2}}-H_{F_{n+1}-1,{-1/2}}=2 \Hn(n)$. 
We substitute these values into
\eqref{HarmonicInequality}. Thus, we first substitute $p=-1/2$ and $m=F_{n+1}-1$ and next we substitute $p=-1/2$ and $m=F_{n+2}-1$. This gives rise to the  inequality: 
$$
\dfrac{(F_{n+2}-1)^{3/2}-(F_{n+1}-1)^{3/2}}{3}< \Hn(n) < \dfrac{F_{n+2}^{3/2}-F_{n+1}^{3/2}}{3}.
$$
This implies   
$$
  \dfrac{(F_{n+3}-1)^{3/2}-(F_{n+2}-1)^{3/2}}{F_{n+2}^{3/2}-F_{n+1}^{3/2}}<\dfrac{\Hn(n+1)}{\Hn(n)}<\dfrac{F_{n+3}^{3/2}-F_{n+2}^{3/2}}{(F_{n+2}-1)^{3/2}-(F_{n+1}-1)^{3/2}}. 
$$
Hence, 
\begin{multline*}
\left(\frac{F_{n+3}}{F_{n+2}}\right)^{3/2}\frac{\left(1+\frac{1}{F_{n+3}}\right)^{3/2}-\left(\frac{F_{n+2}}{F_{n+3}}-\frac{1}{F_{n+3}}\right)^{3/2}}{1-\left(\frac{F_{n+1}}{F_{n+2}}\right)^{3/2}}<\frac{\Hn(n+1)}{\Hn(n)}\\<
\left(\frac{F_{n+3}}{F_{n+2}}\right)^{3/2}\dfrac{1-\left(\frac{F_{n+2}}{F_{n+3}}\right)^{3/2}}{\left(1-\frac{1}{F_{n+2}}\right)^{3/2}-\left(\frac{F_{n+1}}{F_{n+2}}-\frac{1}{F_{n+2}}\right)^{3/2}}.
\end{multline*}
Now, using the squeeze theorem and letting $n$ tend to infinity, we can see that the middle part of the inequality converges to $\varphi^{3/2}$. This completes the proof. 
\end{proof}

 Consider the sequence 
$$T_{n,m}=\sum_{i=1}^{F_{n-1}} \sqrt{F_n+i^m}, \text{ for } m\geq 1.$$ 
The sequence $T_{n,m}$ is an abstract generalization of sequence $H(n)$.

\begin{theorem}\label{Problem6} If $n>0$, then 
$$\lim_{n\to \infty}\frac{T_{n+1,m}}{T_{n,m}}=\sqrt{\frac{1}{2}(L_{m+2}+F_{m+2}\sqrt{5})}.$$
\end{theorem}

\begin{proof}
Let us simplify the expression in the right-hand side. Since 
$$
F_m=\frac{\alpha^m-\beta^m}{\sqrt{5}}\qquad {\text{\rm and}}\quad L_m=\alpha^m+\beta^m,\quad {\text{\rm where}}\quad (\alpha,\beta):=(\varphi,-\varphi^{-1}), 
$$
we have 
$$
{\sqrt{\frac{1}{2}\left(L_{m+2}+{\sqrt{5}}F_{m+2}\right)}}=\alpha^{(m+2)/2}=\varphi^{(m+2)/2}.
$$
We consider two cases for $m$. If $m=1$, then the sum becomes 
$$
T_{n,1}=\sum_{F_{n}<k\le F_{n+1}} {\sqrt{k}}.
$$
We use that the estimate (see \cite{FLuca})
$$
\sum_{k\le x} {\sqrt{k}} =\frac{2}{3} x^{3/2}+O({\sqrt{x}})
$$
holds for all $x\ge 1$.  
Then 
$$
\sum_{F_n<k\le F_{n+1}} {\sqrt{k}}=\sum_{k\le F_{n+1}} {\sqrt{k}}-\sum_{k\le F_n} {\sqrt{k}}=\frac{2}{3}\left(F_{n+1}^{3/2}-F_n^{3/2}\right)+O({\sqrt{F_{n+1}}}).
$$
Since $F_n=\varphi^n/{\sqrt{5}}(1+o(1))$, we have that 
$$
F_{n+1}^{3/2}=\varphi^{3(n+1)/2}/5^{3/4} (1+o(1)),\qquad {\text{\rm and}}\qquad F_n^{3/2}=\varphi^{3n/2}/5^{3/4}(1+o(1)). 
$$
So,
$$
F_{n+1}^{3/2}-F_n^{3/2} =\phi^{3n/2}/5^{3/4}(\varphi^{3/2}-1)(1+o(1)).
$$
This gives that 
$$
T_{n,1}=\varphi^{3n/2}/5^{3/4}(\varphi^{3/2}-1)(1+o(1))+O(\varphi^{n/2})=\varphi^{3n/2}/5^{3/4}(\varphi^{3/2}-1)(1+o(1)).
$$
Therefore, 
$$
\frac{T_{n+1,1}}{T_{n,1}}=\frac{\varphi^{3(n+1)/2}/5^{3/4}(\varphi^{3/2}-1)(1+o(1))}{\varphi^{3n/2}/5^{3/4}(\varphi^{3/2}-1)(1+o(1))}=\phi^{3/2}(1+o(1)),
$$
This implies that 
$$
\lim_{n\to\infty} \frac{T_{n+1,1}}{T_{n,1}}=\varphi^{3/2}=\varphi^{(m+2)/2},
$$
which is what we wanted. 

Case $m\ge 2$. We observe that  
$$
i^{m/2}\le {\sqrt{F_n+i^m}}\le i^{m/2}+{\sqrt{F_n}}\qquad {\text{\rm for~all}}\qquad i\in [1,F_{n-1}].
$$
Thus, 
$$
S_{n,m}\le T_{n,m}\le S_{n,m}+F_{n-1}{\sqrt{F_n}},
$$
where 
$$
S_{n,m}:=\sum_{i=1}^{F_{n-1}} i^{m/2}.
$$
Again by the Abel summation formula, or elementary comparisons between $S_{n,m}$ and $\int_1^{F_{n-1}} x^{m/2}dx$, we get that 
$$
S_{n,m} =\left(\frac{2}{m+2}\right) F_{n-1}^{(m+2)/2}+O_m(F_{n-1}^{m/2}). 
$$
Since $m\ge 2$, $(m+2)/2\ge 2$. So, $F_{n-1}^{(m+2)/2}\gg \varphi^{(m+2)(n-1)/2}$, while 
$$F_{n-1}{\sqrt{F_n}}<\varphi^{3n/2}=o(\varphi^{(m+2)(n-1)/2})\qquad {\text{\rm  as }}\qquad n\to\infty,
$$ 
and $m\ge 2$ is fixed. 
This shows that 
$$
F_{n-1}{\sqrt{F_n}}=o(S_{n,m}),
$$
and therefore 
$$
T_{n,m}=S_{n,m}(1+o(1))=\left(\frac{2}{m+2}\right) F_{n-1}^{(m+2)/2}(1+o(1)).
$$
It follows that 
$$
\frac{T_{n+1,m}}{T_{n,m}}=\frac{(2/(m+2))F_{n}^{(m+2)/2}(1+o(1))}{(2/(m+2))F_{n-1}^{(m+2)/2}(1+o(1))}=\left(\frac{F_{n}}{F_{n-1}}\right)^{(m+2)/2}(1+o(1)).
$$
This implies that 
$$
\lim_{n\to\infty} \frac{T_{n+1,m}}{T_{n,m}}=\left(\lim_{n\to\infty} \frac{F_{n}}{F_{n-1}}\right)^{(m+2)/2}=\varphi^{(m+2)/2},
$$
which is what we wanted. 
\end{proof}

\section{Angles of the Fibonacci--Theodorus spiral}

In this section, we investigate the angular properties of the Fibonacci--Theodorus spiral. Notably, we provide precise bounds and asymptotic approximations for these angles. However, it is important to acknowledge that a few questions remain open. 

For example, while we demonstrate that the number of triangles required for a full revolution is 10, a particular question remains open: is it possible to find two hypotenuses that lie on the same straight line? This is one of the questions that we seek to address within the realm of the Fibonacci-Theodorus spiral. 

\begin{proposition}\label{PFS1} For the spiral depicted in Figure \ref{FibonacciTheodorusSpiral}, the following statements hold:
\begin{enumerate}
    \item \label{FSP1} $\lim_{n\to\infty}\tan(\theta_n)=\sqrt{\left|\beta\right|}$. 
	\item \label{FSP2} $\lim_{n\to\infty}\theta_n=\arctan(\sqrt{|\beta|})$, which can be expressed as $$\sqrt{\left|\beta\right|}  \sum_{n=1}^\infty\frac{(-1)^n(L_n-F_n\sqrt{5})}{2n+1}.$$ 
    \item \label{FSP3} The angle $\theta_n$ satisfies the inequality 
 $$\arctan\left(\sqrt{\frac{1}{2}}\right)\leq \theta_n\leq \frac{\pi}{4}.$$
    \item \label{FSP4} It takes precisely $10$ triangles to complete one full winding.
\end{enumerate}
	\end{proposition}
	
	\begin{proof} To prove Part \eqref{FSP1}, observe from Figure \ref{FibonacciTheodorusSpiral} that $\tan\left(\theta_n\right)=\sqrt{F_{n}/F_{n+1}}$. It is well-known that $\lim_{n\to\infty}F_{n+1}/F_n =\varphi$ (see, for example, \cite{Vajda}). Thus, the conclusion of Part \eqref{FSP1} follows from $$\lim_{n\to\infty}\frac{F_{n}}{F_{n+1}}=\frac{1}{\varphi}=\left|\beta\right|.$$

We prove \eqref{FSP2}. According to Vajda \cite{Vajda}, we know that $\left|\beta\right|^n= (L_n-F_n\sqrt{5})/2$. Using the power series for 
$$\arctan (x)=\sum_{n = 1}^\infty (-1)^{n } \frac{x^{2n+1}}{2 n+1},$$ since $-1\leq\sqrt{|\beta|}\leq 1$, we obtain our desired result.

To prove \eqref{FSP3}, we start with Binet's formula and simplify to obtain 
		$$\begin{aligned} \frac{F_{n+1}}{F_n} 
    =\frac{\varphi(\varphi^n-\beta^n)+\varphi\beta^n-\beta^{n+1}}{\varphi^n-\beta^n}
         =\varphi+\beta^n\left(\frac{\varphi-\beta}{\varphi^n-\beta^n}\right)
  	=\varphi+\left(\frac{\varphi-\beta}{(-\varphi^2)^n-1}\right).
   \end{aligned}$$
   This implies     
		$$\frac{F_{n+1}}{F_n}=
        \begin{cases}
			\varphi+\left(\frac{\sqrt{5}}{\varphi^{2n}-1}\right), \text{ if $n$ is even}\\
			\varphi-\left(\frac{\sqrt{5}}{\varphi^{2n}+1}\right),\text{ if $n$ is odd}.
		\end{cases}$$
This gives upper and lower bounds for $F_{n+1}/F_n$, with these values occurring when $n=2$ and $n=1$ in $(-\varphi^2)^n-1$. Therefore, we have $1\leq F_{n+1}/F_n\leq 2$. This implies that $\sqrt{1/2}\leq \sqrt{F_{n-1}/F_n}\leq 1$. As a result, we have: 
$\arctan(\sqrt{1/2})\leq \theta_n\leq \pi/4$.
			
We now prove Part \eqref{FSP4}. Since $F_{n+1}/F_n$ converges to $\varphi$ and $\theta_n$ converges to $\arctan(\sqrt{|\beta|})$ as $n$ approaches $\infty$ (see Part \eqref{FSP2}), we can use the definition of the limit: for any $\epsilon > 0$, there is an integer $N$ such that for $n\geq N$,  
$|\arctan (\theta_n)-\arctan (\sqrt{|\beta|})|\leq\epsilon$. 
Approximating the values of $\theta_n$, we have  
$36^\circ<\theta_n<40^\circ$ for $n\geq 3$. Note that $\theta_2\approx 35.264^\circ<36^\circ$ and $\theta_1$ is $45^\circ$. So, 
		$360^\circ<\sum_{i=0}^9\theta_{n+i}<400^\circ$.	
	Hence it always takes ten triangles to complete one revolution.
	\end{proof}
	
In the introduction, we analyze how the distance between windings of the Theodorus spiral behaves as $n$ approaches $\infty$. We observed that the distance between windings of the Theodorus spiral tends towards $\pi$. However, the Fibonacci--Theodorus spiral, unlike the Theodorus spiral, does not approach $\pi$. Since it takes ten triangles to complete a winding, we have $\lim_{n\to\infty}{\sqrt{F_{n+10}}-\sqrt{F_n}}=\infty$. Therefore, the distance between windings increases without bounds.

\subsection{Geometric Results in the Fibonacci-Theodorus Spiral}
We now present a sequence of three results inspired by classic geometric concepts but naturally aligned with the topics of this paper. 

In the traditional geometry, we know that the \emph{incircle} or \emph{inscribed circle} of a given triangle is the largest circle that can be contained in the triangle. A triangle is \emph{circumscribed} by the circle $C$ when it passes through all three vertices in the triangle. Define $C_n$ as the circumscribed circle of $\Delta_n$, see Figure \ref{winf2} (left). We define $T_n$ as the incircle or inscribed circle of $\Delta_n$, as depicted in Figure \ref{winf2} (right).

\begin{figure}[ht]
\centering
 \includegraphics[scale = 0.65]{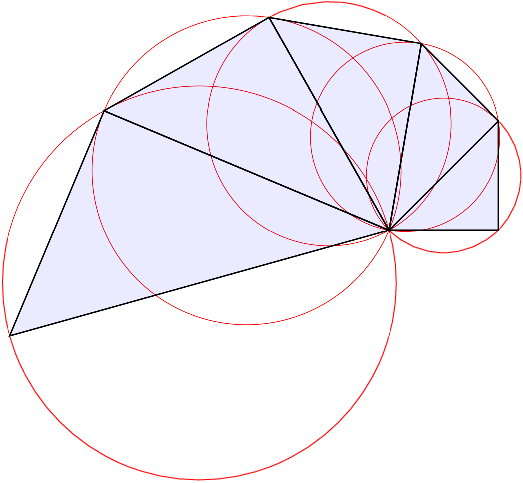} \hspace{2cm}
 \includegraphics[scale = 0.75]{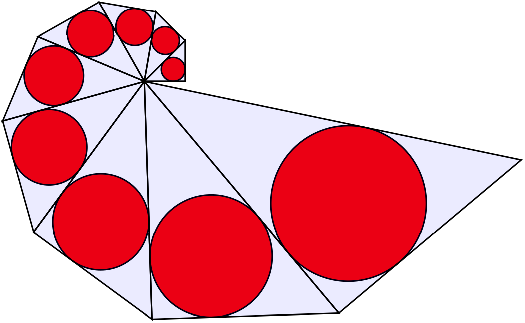}
\caption{Circumscribed and inscribed circles of $\Delta_n$.} 
	 \label{winf2}
\end{figure}

\begin{theorem}
 For $n\ge 1$, holds that the area of $C_n$ is  $(\pi/4) F_{n+2}$ and the perimeter of $C_n$ is $\pi F_{n+2}$.
\end{theorem}

\begin{proof}
Let $R_n$ be radius of $C_n$. Then $R_n={\sqrt{F_{n+2}}}/2$ --the radius of the circumscribed circle of a right triangle is $1/2$ of the hypothenuse. 
This implies that  ${{\text{\rm Area}}}(C_n)=\pi R_n^2$ and ${{\text{\rm Perimeter}}}(C_n)=2\pi R_n$. 
\end{proof}

\begin{theorem} 
 For $n\ge 1$, these hold 
\begin{enumerate}
\item \label{IncircleRn}
$$
\lim_{n\to\infty} \frac{A_n}{{\text{\rm Area}}(T_n)}=\frac{2{\sqrt{\varphi}}}{\pi (1+{\sqrt{\varphi}}-\varphi)^2}.
$$
 \item  \label{CircumscribedCn}
$$
\lim_{n\to\infty} \frac{{\text{\rm Area}}(C_n)}{{\text{\rm Area}}(T_n)}=\frac{\varphi^2}{(1+{\sqrt{\varphi}}-\varphi)^2}.
$$
\end{enumerate}
\end{theorem}

\begin{proof} We prove Part \eqref{IncircleRn}. 
First of all, we recall that $R_n={\sqrt{F_{n+2}}}/2$. Let $r_n$ be the radius of the inscribed circle. It is known that in right triangle $(a_n,b_n,c_n)$, it holds that 
$
r_n=\frac{1}{2}(a_n+b_n-c_n),
$
where $a_n$ and $b_n$ are the lengths of the legs and $c_n$ is the length of the hypothenuse.
This applied to our case gives  $r_n  =  \frac{1}{2} ({\sqrt{F_n}}+{\sqrt{F_{n+1}}}-{\sqrt{F_{n+2}}})$. Therefore, 
\begin{eqnarray*}
r_n & = & \frac{1}{2 \cdot 5^{1/4}} (\varphi^{n/2}(1+o(1))+\varphi^{(n+1)/2}(1+o(1))-\varphi^{(n+2)/2}(1+o(1)))\\
& = & 
\frac{1}{2\cdot 5^{1/4}}(1+{\sqrt{\varphi}}-\varphi) \varphi^{n/2}(1+o(1)).
\end{eqnarray*}
So, an asymptotic for ${\text{\rm Area}}(T_n)$ is given by 
$$
\pi r_n^2=\frac{\pi}{4{\sqrt{5}}}(1+{\sqrt{\varphi}}-\varphi)^2\varphi^{n}(1+o(1)).
$$
Certainly,
$$
A_n=\frac{\sqrt{F_nF_{n+1}}}{2}=\frac{1}{2{\sqrt{5}}} \varphi^{n/2}(1+o(1))\varphi^{(n+1)/2}(1+o(1))=\frac{{\sqrt{\varphi}}}{2{\sqrt{5}}} \varphi^{n}(1+o(1)).
$$
So, we see that
$$
\frac{A_n}{{\text{\rm Area}}(T_n)}=\frac{{\sqrt{\varphi}}/(2{\sqrt{5}})\varphi^n(1+o(1))}{\pi/(4{\sqrt{5}})(1+{\sqrt{\varphi}}-\varphi)^2 \varphi^n (1+o(1))}=\frac{2{\sqrt{\varphi}}}{\pi(1+{\sqrt{\varphi}}-\varphi)^2}(1+o(1)).
$$
This implies that 
$$
\lim_{n\to\infty} \frac{A_n}{{\text{\rm Area}}(T_n)}=\frac{2{\sqrt{\varphi}}}{\pi (1+{\sqrt{\varphi}}-\varphi)^2}.
$$

We prove Part \eqref{CircumscribedCn}. The area of $C_n$ is given by 
$$
\pi R_n^2=\pi \left(\frac{F_{n+2}}{4}\right)=\frac{\pi\varphi^{n+2}}{4{\sqrt{5}}}(1+o(1))=\frac{\pi \varphi^2}{4{\sqrt{5}}} \varphi^n (1+o(1)).
$$
Hence,
$$
\frac{{\text{\rm Area}}(C_n)}{{\text{\rm Area}}(T_n)}=\frac{\pi R_n^2}{\pi r_n^2}=\frac{\pi \varphi^2/(4{\sqrt{5}}) \varphi^n(1+o(1))}{\pi/(4{\sqrt{5}})(1+{\sqrt{\varphi}}-\varphi)^2 \varphi^n(1+o(1))}=\frac{\varphi^2}{(1+{\sqrt{\varphi}}-\varphi)^2} (1+o(1)). 
$$
This implies that 
$$
\lim_{n\to\infty} \frac{{\text{\rm Area}}(C_n)}{{\text{\rm Area}}(T_n)}=\frac{\varphi^2}{(1+{\sqrt{\varphi}}-\varphi)^2}.
$$
This completes the proof.
\end{proof}

\begin{problem}
How many triangles are needed to complete two and three windings in the Fibonacci spiral?
\end{problem}

\section{Some generalizations}

In this final section, we introduce a general spiral, which we term the \emph{Sequence--Theodorus Spiral}. The construction process is as illustrated in Figure \ref{FibonacciTheodorusSpiral}. To initiate this construction, consider a second-order recurrence relation: $a_{n+1}=a_n+a_{n-1}$ for $n\geq 2$. Here, we set $a_1=a$ and $a_2=b$. In case $a_0 \ne 0$, we can make the appropriate adjustment by simply shifting the sequence by one. The construction involves a specific triangle, as illustrated on the left graph and in the middle graph of Figure ~\ref{fig:FT2GenCase}. For $n\geq 2$, this triangle has an adjacent side of $\sqrt{a_{n}}$, an opposite side of $\sqrt{a_{n-1}}$, and a hypotenuse of $\sqrt{a_{n+1}}$.

 \begin{figure}[ht]
	\centering
 \includegraphics[width=3.5in]{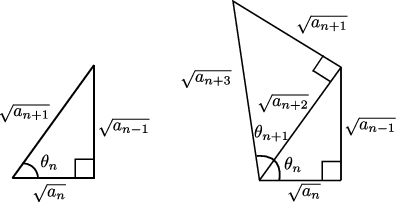} \hspace{.5cm}
	\includegraphics[scale = 0.35]{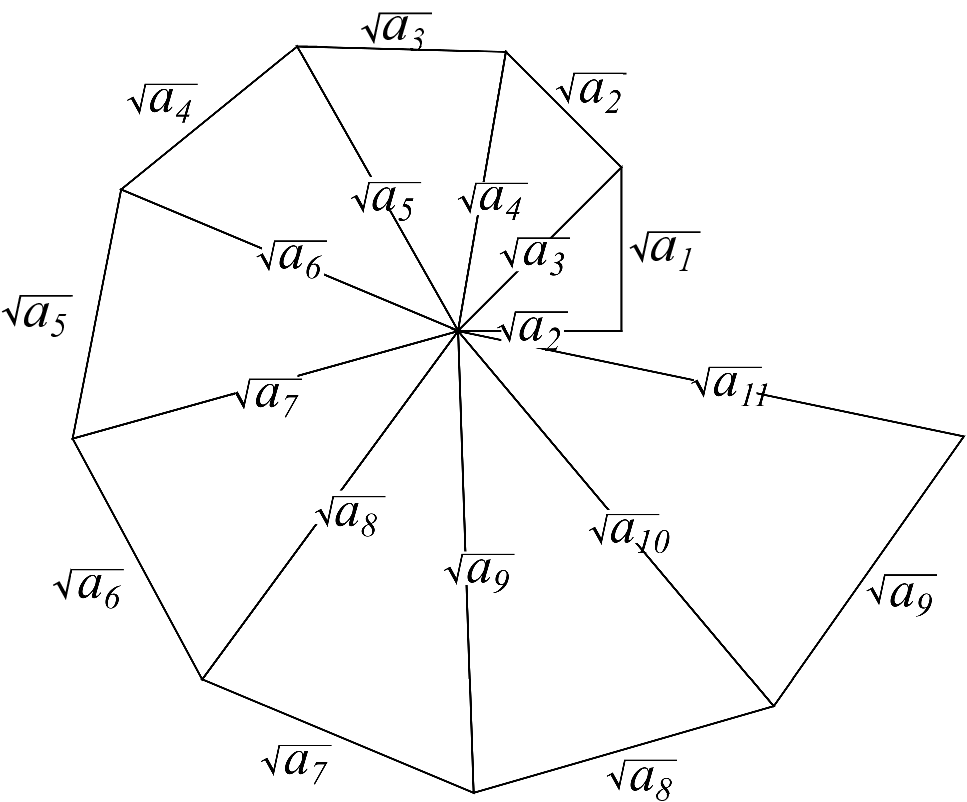}
	\caption{Generalized Theodorus Spiral.}
	\label{fig:FT2GenCase}
	\end{figure}

\subsection{A problem}
By connecting each of these triangles, as demonstrated in the figure, we achieve the desired Sequence--Theodorus Spiral. The right-hand side of the figure provides a visualization of the Sequence--Theodorus Spiral's construction, emphasizing its foundation in the general construction method. This method can be used to generate spirals for second-order recurrence relations. For instance, if we set $a_n = L_n$, we can aptly refer to it as the \emph{Lucas--Theodorus Spiral}. Similarly, we can create the \emph{Pell--Theodorus Spiral} or the \emph{Jacobsthal--Theodorus Spiral}, and so on. See also a second construction done by Crilly in \cite{Crilly}.

\begin{problem}
Explore the geometric results discussed in this paper for the Fibonacci sequence, adapting them to the specific sequence considered in this section.
\end{problem}

\medskip

\end{document}